\documentclass[reqno]{amsart}

\usepackage{amsmath}%
\usepackage{amsfonts}%
\usepackage{amssymb}%
\usepackage{graphicx}
\usepackage[hidelinks]{hyperref,xcolor}

\definecolor{gray}{rgb}{0.2,0.2,.2}
\hypersetup{
    colorlinks=true,
    citecolor=blue,
		linkcolor=gray
}

\newtheorem{theorem}{Theorem}
\theoremstyle{plain}

\newtheorem{lemma}{Lemma}

\newtheorem{proposition}{Proposition}
\newtheorem{remark}{Remark}

\numberwithin{equation}{section}

\newcommand{\chain}[1]{\begin{center} \ensuremath{\{#1\}} \end{center}}
\newcommand{\bb}[1]{\ensuremath{\mathbb{#1}}}
\newcommand{\bprod}[2]{\ensuremath{\prod_{#1}^{#2}}}
\newcommand{\bsum}[2]{\ensuremath{\sum_{#1}^{#2}}}

\begin{document}
\title[Nonexistence of OPN]{More on the Nonexistence of Odd Perfect Numbers of Certain Forms}
\author{Patrick A. Brown}
\address{Dr. Brown is not currently affiliated with any institution. Feel free to \newline
\indent email him regarding any questions on the data prepared for this paper.}
\email{PatrickBrown496@gmail.com}
\urladdr{}

\thanks{Thanks go to Kevin G. Hare for endorsing me so that this paper could appear on ArXive.org. His endorsement goes only so far as to admit that this work might look like real mathematics. Any mistakes are the sole responsibility of the author.}

\date{November 2015}
\subjclass[2010]{11-04, 11A25, 11B83, 11Y55}
\keywords{Odd Perfect Numbers}
\dedicatory{}

\begin{abstract}
Euler showed that if an odd perfect number exists, it must be of the form $N = p^\alpha q_{1}^{2\beta_{1}} \ldots q_{k}^{2\beta_{k}}$, where $p, q_{1}, \ldots, q_k$ are distinct odd primes, $\alpha$, $\beta_{1}$, \ldots , $\beta_{k} \in \bb{N}$, with $p \equiv \alpha \equiv 1 \pmod{4}$. In 2005, Evans and Pearlman showed that $N$ is not perfect, if $3|N$ or $7|N$ and each $\beta_{i} \equiv 2 \pmod{5}$. We improve on this result by removing the hypothesis that $3|N$ or $7|N$ and show that $N$ is not perfect, simply, if each $\beta_{i} \equiv 2 \pmod{5}$.
\end{abstract}

\maketitle

\section{Introduction}

We define $\sigma(N)$ to be the sum of the positive divisors of $N$. We say $N$ is perfect when $\sigma(N) = 2N$. For example, $\sigma(6) = 1 + 2 + 3 + 6 = 12$ and $\sigma(28) = 1 + 2 + 4 + 7 + 14 + 28 = 56$, making both $6$ and $28$ perfect. It is still an open question as to whether or not there exists an infinite number of even perfect numbers or even a single example of an odd perfect number. Nevertheless, we let \bb{O} be the set of all odd perfect numbers.

Euler showed that if an odd perfect number exists, it must be of the form:

\begin{equation} \label{form}
N = p^\alpha q_{1}^{2\beta_{1}} \ldots q_{k}^{2\beta_{k}}
\end{equation}

\noindent where $p, q_{1}, \ldots, q_k$ are distinct odd primes, $\alpha$, $\beta_{1}$, \ldots , $\beta_{k} \in \bb{N} = \{0, 1, 2, \ldots\}$, with $p \equiv \alpha \equiv 1 \pmod{4}$. The prime $p$ is often referred to as the \emph{special} prime of $N$, and $p^\alpha$ as the \emph{Eulerian} component of $N$. Throughout this paper, when we say $N \in \bb{O}$, unless otherwise stated, we assume $N$ has the form given in \ref{form}.

Assuming $\beta_{1} = \cdots = \beta_{k} = \beta$, it has been shown for all fixed $\beta \leq 14$, except $\beta = 9$, that $N$ cannot be odd perfect. Additional results include infinite congruence classes for $\beta$. For example, McDaniel proved in \cite{Mcdan1970} that $N$ is not perfect if each $\beta_{i} \equiv 1 \pmod{3}$. Iannucci and Sorli, in \cite{ian2003}, showed $N$ is not perfect if $3|N$ and each $\beta_{i} \equiv 1 \pmod{3}$ or $\beta_{i} \equiv 2 \pmod{5}$. See Evans and Pearlman, \cite{evans2005} for a more detailed account of these types of results. In that same paper, Evans and Pearlman show

\begin{theorem} \label{premain}
Suppose $N \in \bb{O}$ and each $\beta_{i} \equiv 2 \pmod{5}$, then gcd$(N,21) = 1$ and $p \equiv 1 \pmod{12}$.
\end{theorem}

Using this result, and applying a different method of proof, we show

\begin{theorem} \label{main}
Suppose $N$ is as described in (\ref{form}) and each $\beta_{i} \equiv 2 \pmod{5}$, then $N$ is not perfect.
\end{theorem}
\noindent Theorem \ref{main} subsumes a number of theorems like Theorem \ref{premain}.

\section{Preliminary Results}

We start with some elementary properties of $\sigma$:

\begin{enumerate}
\item if $q$ is prime, $k \in \bb{N}$, then $\sigma(q^{k}) = 1 + q + q^{2} + \ldots + q^{k}$. 
\item $\sigma$ is multiplicative. That is, if gcd$(m,n) = 1$, then $\sigma(m*n) = \sigma(m) * \sigma(n)$.
\end{enumerate}

\noindent For $N \in \bb{O}$:

\begin{equation}
\sigma(N) = \sigma(p^\alpha) \sigma(q_{1}^{2\beta_{1}}) \ldots \sigma(q_{k}^{2\beta_{k}}) = 2N
\end{equation}

\noindent Which makes clear that any prime $r$ dividing $\sigma(q_{i}^{2\beta_{i}})$ for some $i$, must also divide $N$. Additionally, $\alpha$ odd implies that for $\alpha = 2a+1$

\begin{eqnarray}
\sigma(p^\alpha) & = & 1+p+ \dots + p^{2k} + p^{2a+1} \nonumber \\
& = & (1+p) + \ldots + (1+p)p^{2a} \nonumber \\
& = & (1+p)(1 + p^{2} + \dots + p^{2a}) \label{p+1}
\end{eqnarray}

\noindent Thus, $(p+1)|\sigma(p^\alpha)$ and so any odd prime dividing $(p+1)$ also divides $N$.

We will make use of the following lemmas. The proof of the first can be found in \cite{pearl2005} and the second in \cite{kanold1941}.

\begin{lemma} \label{Pearlman2005}
$\sigma(s^{f})|\sigma(s^{f+(f+1)m})$ for all primes $s$ and all $m,f \in \bb{N}$.
\end{lemma}

Let $N \in \bb{O}$ be written as in (\ref{form}). Define $\gamma_{i} = 2\beta_{i} + 1$ for $1 \leq i \leq k$.

\begin{lemma} \label{Kanold1941}
Suppose $d|\gamma_{i}$ for each i, then $d^{4}|N$.
\end{lemma}

\begin{remark} \label{chainStarter}
Observe the assumption in Theorem \ref{main} that each $\beta_{i} \equiv 2 \pmod{5}$ is equivalent to assuming $d=5$ in Lemma \ref{Kanold1941}.
\end{remark}

\noindent Further note, with $f=4$ and $2\beta_{i} \equiv 4 \pmod{5} = 5m+4$ for some $m \in \bb{N}$, Lemma \ref{Pearlman2005} implies 
\begin{equation}
\sigma(s^{4})|\sigma(s^{5m+4}) = \sigma(s^{2\beta_{i}}).
\end{equation}
\noindent Though Lemma \ref{Pearlman2005} is often discussed in the context of cyclotomic polynomials, we won't need results any more powerful than the above. These two ideas lie at the foundation of every result similar to Theorem \ref{premain} and Theorem \ref{main}.

\section{$\sigma$-Chains}

Under the hypotheses of Theorem \ref{main}, suppose $r_{0}$ is a prime that divides $\sigma(q_{i}^{4})$. It follows immediately that

\begin{equation}
r_{0}|\sigma(q_{i}^{4})|\sigma(q_{i}^{2\beta_{i}})|N \Rightarrow r_{0}|N.
\end{equation}

Thus, if $r_{0} \not \equiv 1 \pmod{4}$, then $r_{0}$ cannot be the special prime in the prime factorization of $N$, so there must be some $q_{i}$ where $r_{0} = q_{i}$.
\begin{remark} \label{specPrime}
Applying Theorem \ref{premain}, we can reach this same conclusion so long as $r_{0} \not \equiv 1 \pmod{12}$. Additionally, since we know $7 \not |N$, (\ref{p+1}) implies $r_{0} \not \equiv 6 \pmod{7}$.
\end{remark}
\noindent For the rest of this paper, we assume our special prime $p$ satisfies these two conditions.

Applying the same reasoning to a prime, say $r_{1}$, such that $r_{1}|\sigma(r_{0}^{4})$, so long as $r_{1}$ does not satisfy the conditions of Remark \ref{specPrime}, then we may conclude $r_{1}|N$. Continuing this process, we may construct a chain $\{r_{0}, r_{1}, r_{2}, \ldots \}$ of primes satisfying $r_{i+1}|\sigma(r_{i}^{4})$ whereby each $r_{i+1}$ must divide $N$.

Traditionally, these chains have been used to accumulate enough primes to show that $\sigma(N) > 2$, thus contradicting $\sigma(N) = 2$. We will be using them to create a contradiction on a different internal structure of $N$. For the moment, we will use them to show the following

\begin{proposition} \label{propT}
If $N \in \bb{O}$ satisfies the hypotheses of Theorem \ref{main}, each prime in the collection $\bb{T} = \{5, 11, 31, 41, 71, 101, 131, 151, 181, 191, 211\}$ divides $N$.
\end{proposition}
\begin{proof}
By Remark \ref{chainStarter}, we know $5|N$, which we use to construct our first two chains:
\chain{5, 11}
\chain{5, 71, 211, 1361, 11831, 17249741, 41}
Once we know $211|N$, we may construct the chain:
\chain{211, 292661, 191, 13001, 32491, 34031, 101, 31}
\noindent and after $191|N$,
\chain{191, 1871, 151}
\chain{191, 13001, 17981, 613680341, 1478611, 520392931, 336491, 4231, 216211, 131}
\chain{191, 13001, 32491, 34031, 350411, 47791, 26561, 181}
A quick scan through each chain shows that none of the primes are congruent to $1 \pmod{12}$, except 181 which is also $6 \pmod{7}$. Thus, none of the primes can be the special prime, and since each element of \bb{T} appears, we are done.
\end{proof}

\begin{proposition} \label{multi}
For the same $N$, each prime in \bb{T} divides $N$ at least $5$ times.
\end{proposition}
\begin{proof}
We skip most of the details and simply point out that
\begin{center} $11\ $ divides $\ \sigma(5^{4}), \sigma(31^{4}), \sigma(71^{4}), \sigma(191^{4})$, and $\sigma(311^{4})$ \end{center}
\begin{center} $31\ $ divides $\ \sigma(101^{4}), \sigma(281^{4}), \sigma(1031^{4}), \sigma(1151^{4})$, and $\sigma(1721^{4})$ \end{center}
\noindent and similarly, every element of \bb{T} can be shown to appear in at least $5\ \sigma$-chains of various primes under $10000$, except 181, which requires us to go to as high as 18521.
\end{proof}

For the rest of the paper we expand the definition of \bb{T} to include all of the primes demonstrated through $\sigma$-chains to divide $N$. Suppose now that it has been demonstrated $\{ 821, 55001\} \subset \bb{T}$. Consider the following two chains:
\chain{821, 241, 61}
\chain{55001, 2521, 61}
Observe that $61, 241,$ and $2521$ each satisfy the conditions of Remark \ref{specPrime} and are candidates for being the special prime. Neither chain has any primes in common except for the last element, 61. At most, one of 241 or 2521 may be the special prime. By virtue of this fact, one of the two chains shows $61 \in \bb{T}$.

\section{A Closer Look at the Form of $N$} \label{ewellForm}

In 1980, Ewell \cite{ewell1980} demonstrated the following:
\begin{theorem} \label{Ewell}
Let $N = p^\alpha q_{1}^{2\beta_{1}} \ldots q_{k}^{2\beta_{k}}$ be odd perfect. Set $\alpha = 4\epsilon+1$ and $q_{i} = 2\pi_{i}-1$. Assume $3 \not | N$, then
\begin{enumerate}
\item $p \equiv 1 \pmod{12}$,
\item $\epsilon \equiv 0$ or $-1 \pmod{3}$,
\item $\pi_{i}\beta_{i} \equiv 0$ or $-1 \pmod{3}$ for each $i$,
\item The number of elements in the set $\{\pi_{1}\beta_{1}, \ldots, \pi_{k}\beta_{k}\}$ for which $\pi_{i}\beta_{i} \equiv -1 \pmod{3}$ is even when $\epsilon \equiv 0 \pmod{3}$ and odd when $\epsilon \equiv -1 \pmod{3}$.
\end{enumerate}
\end{theorem}

Applying our assumption $\beta_{i} \equiv 2 \pmod{5}$ to this result, simple algebra shows that $N$, after some relabeling, may be written as
\begin{equation}
N = p^\alpha \bprod{i=1}{s} q_{i}^{10\beta_{i}+4} \bprod{j=1}{t} r_{j}^{30\gamma_{j}+4} \bprod{k=1}{u} s_k^{30\delta_{k}+24}
\end{equation}
\noindent Where $p, q_{i}, r_{j}$, and $s_{k}$ are primes with $p \equiv 1 \pmod{12}$, each $q_{i} \equiv -1 \pmod{6}$, and each $r_{j}$,$s_{k} \equiv 1 \pmod{6}$ and, for their respective subscripts, each $\beta$, $\gamma$, $\delta$, and $\alpha \in \bb{N}$ with $\alpha \equiv 1$ or $9 \pmod{12}$.

This form of $N$ is a bit cumbersome, but it does clearly demonstrate one useful piece of information.
\begin{remark} \label{multiHigher}
From Proposition \ref{multi}, it is immediate that $5|N$ at least 14 times, $11|N$ at least 14 times, $31|N$ at least 24 times, and so on, according to each prime's residue (mod 6).
\end{remark}

\section{The Sum of the Reciprocal of the Primes of $N$} \label{upperBound}
We now consider \cite{cohen1978}. Cohen demonstrates the following for odd perfect numbers such that $5|N$ and $3 \not |N$:
\begin{equation}
\bsum{q|N}{} \frac{1}{q} < .677637 \nonumber
\end{equation}

\noindent To this author's knowledge, the best upper and lower bounds of these types appear in \cite{cohen1980}, and we include them here.

\begin{eqnarray}
Divisors & Lower Bound & Upper Bound \nonumber \\
3\not |N, 5|N & .647649 & .677637 \nonumber \\
3\not |N, 5 \not |N & .667472 & .693148 \nonumber \\
3 |N, 5 |N & .596063 & .673634 \nonumber \\
3\not |N, 5 \not |N & .604707 & .657304 \nonumber
\end{eqnarray}

Given that we have additional knowledge regarding the structure of $N$, it is natural we try to improve on this bound for our particular case. Cohen starts his proof off by demonstrating a stronger form of the following inequality, for $0 < x \leq \frac{1}{3}$

\begin{equation} \label{CohenIneq}
1+x+x^2 > \exp(x)
\end{equation}

\noindent We assume (\ref{CohenIneq}) and let $N \in \bb{O}$. Thus,
\begin{equation} \nonumber
2N = \sigma(N) = (1+p+p^{2}+ \ldots + p^{\alpha}) \bprod{i=1}{k} (1+q_{i}+q_{i}^{2}+ \ldots + q_{i}^{2\beta_{i}})
\end{equation}

\noindent Dividing through by $N$ and utilizing $\alpha \geq 1$ yields,
\begin{equation} \nonumber
2 \geq (1+\frac{1}{p}) \bprod{i=1}{k} (1+\frac{1}{q_{i}}+\frac{1}{q_{i}^{2}}+ \ldots + \frac{1}{q_{i}^{2\beta_{i}}})
\end{equation}

We separate the $q_{i}$ based on whether or not they appear in \bb{T}. As per Remark \ref{multiHigher}, we let $\gamma_{i}$ be 14 or 24 depending on whether $q_{i} \equiv -1 \pmod{6}$ or $q_{i} \equiv 1 \pmod{6}$, respectively for $q_{i} \in \bb{T}$. And finally, we truncate for the primes not in \bb{T}.
\begin{equation} \nonumber
2 > (1+\frac{1}{p}) \bprod{q_{i} \in \bb{T}}{} (1+\frac{1}{q_{i}}+\frac{1}{q_{i}^{2}}+ \ldots + \frac{1}{q_{i}^{\gamma_{i}}}) \bprod{q_{i} \not \in \bb{T}}{} (1+\frac{1}{q_{i}}+\frac{1}{q_{i}^{2}})
\end{equation}

\noindent Apply (\ref{CohenIneq}) to the components of $q_{i} \not \in \bb{T}$ and take the log,
\begin{equation} \label{2bSubbed}
\ln{2} > \ln{(1+\frac{1}{p})} + \bsum{q_{i} \in \bb{T}}{} \ln{(1+\frac{1}{q_{i}}+\frac{1}{q_{i}^{2}}+ \ldots + \frac{1}{q_{i}^{\gamma_{i}}})} + \bsum{q_{i} \not \in \bb{T}}{} \frac{1}{q_{i}}
\end{equation}

\noindent Let
\begin{equation} \nonumber
\Delta = \bsum{q_{i} \in \bb{T}}{} \ln{(1+\frac{1}{q_{i}}+\frac{1}{q_{i}^{2}}+ \ldots + \frac{1}{q_{i}^{\gamma_{i}}})} - \frac{1}{q_{i}}
\end{equation}

\noindent We substitute $\Delta$ into (\ref{2bSubbed}), add $\frac{1}{p}$ to both sides, and rearrange,
\begin{equation} \label{firstBound}
\ln{2} - \Delta - \ln{(1+\frac{1}{p})} + \frac{1}{p} > \frac{1}{p} + \bsum{i=1}{k} \frac{1}{q_{i}}
\end{equation}

We now observe the right hand side of the inequality is the sum of the reciprocal of the primes that divide $N$. Because $\Delta$ is a straight forward calculation, it would seem our upper bound is singularly dependent upon $p$, however, we can still do a little better by adding a few more primes to \bb{T}. Remark (\ref{specPrime}) implies $p = 37, 61, 73$, or $p \geq 109$ are the only options for the special prime. It will be convenient for us to refer to $\gamma_{q}$ for the exponent of a particular prime $q \in \bb{T}$ as we consider each case in turn.
\begin{enumerate}
\item $p=37$, then $19|(p+1)$ (and thus $N$). Additionally, we know 61 is not the special prime and since we have seen two chains (with appropriate assumptions) that demonstrate $61|N$ we can put $19, 61 \in \bb{T}$ with $\gamma_{19} = \gamma_{61} = 4$.

\item $p=61$, then $31|(p+1)$, which is already in \bb{T}.

\item $p=73$, then $37|(p+1)$, and again, we can put $37, 61 \in \bb{T}$ with $\gamma_{37} = \gamma_{61} = 4$.
\item $p \geq 109$. As before, we can put 61 $\in$ \bb{T} with $\gamma_{61} = 4$.
\end{enumerate}
When $p \geq 109$ we must change our bound slightly. By virtue of the Maclaurin series of $\ln{(1+x)}$, $\ln{(1+x)} > x-\frac{x^{2}}{2}$ for small $x$. Now, (\ref{firstBound}) may be written as
\begin{equation} \nonumber
\ln{2} - \Delta + \frac{1}{2p^{2}} > \frac{1}{p} + \bsum{i=1}{k} \frac{1}{q_{i}}
\end{equation}
After computing the different incarnations of $\Delta$ for each case, we get the following upper bounds for the right hand side of (\ref{firstBound}):

\begin{eqnarray}
& Upper Bound \nonumber \\
p = 37 & .6633150 \nonumber \\
p = 61 & .6646602 \nonumber \\
p = 73 & .6644488  \nonumber \\
p \geq 109 & .6644335 \nonumber
\end{eqnarray}
As can be seen, .6646602 is the greatest of these bounds, so we take this to be, and will refer to it as, \emph{the} upper bound for our $N$.

This may seem like a lot of work to lower the previous upper bound by less than .013. However, when one considers \emph{optimal} solutions, by which we mean, every allowable prime in \bb{T} without concern for special primes. It turns out that \bb{T} must contain 100,369 primes; 5 and every $1 \pmod{5}$ prime less than 5,826,451. This improved bound requires 5 and every $1 \pmod{5}$ prime less than 2,647,111 or a mere 48,250 primes. Of course, the primes that appear in $\sigma$-chains is far from optimal and about 20\% of them are special primes, which as we have seen, require extra attention.

\section{Programming Methodology}

It should be noted that $\sigma$-chains as a term is a bit misleading. A more appropriate term would be $\sigma$-trees. Of course, presenting such information in tree form becomes unwieldly. Though as chains get progressively longer, they too become unwieldly. Considering we had to construct chains to include over 960,000 primes, even after taking great pains to reduce the upper bound in Section \ref{upperBound}, the approximately 10,000 pages of data is likely why Theorem \ref{main} has not been demonstrated sooner.

In an effort to simplify the process of constructing chains, we started with chains of length 2:
\begin{center} \{\emph{Known Seed Prime}\} $\to$ \{\emph{New Prime}\} \end{center}
Chains resulting in a potential special prime or a composite too large for Mathematica to conveniently factor were set aside\footnote{The chains involving large composites were never used. We were able to generate enough primes to prove Theorem \ref{main} without factoring any of these numbers.}. The new primes were put into a ``seed'' list. The next ordered pair, or 2-chain, was created using the smallest number from the seed list. This is not a new idea. The advantage is always working with the smallest numbers from the full $\sigma$-tree, thus avoiding the need to factor large numbers to continue a chain. More importantly, keeping the data as uniform 2-chains made programming chain manipulations much easier to automate and separate between two computers when convenient.

Breaking the upper bound without including candidates for the special prime, though technically feasible, would probably take years. The 2-chains with special primes, previously set aside, were extended to 3-chains where the 3rd element was also a special prime, as follows
\begin{center} \{\emph{Seed Prime}\} $\to$ \{\emph{Special Prime}\} $\to$ \{\emph{Special Prime}\} \end{center}

This collection of 3-chains was then sorted by the last element. The first two appearances of a special prime in the 3rd element were paired together and used to confirm that prime was an element of \bb{T}. After removing chains with duplicate terminus (keeping one chain from the confirmed primes), the 3-chains were extended to 4-chains, again, only extending in the cases where a special prime appeared as the 4th element (that had not already been confirmed in the previous step). The last element of each 4-chain was then compared to the 3rd and 4th element of each of the other 4-chains. Any matches were used to confirm the corresponding special prime was in \bb{T}. This process is not optimal. Among other considerations, we are allowed to have non-special primes appear in these extended chains, however, a broader approach turned out to be unnecessary.

The calculations needed to create the primes included in our set \bb{T}, required 2 laptops a little over 3 weeks to perform. The previous month had been spent trying to work with chains of arbitrary length. As aforementioned this becomes quite cumbersome. Once this 2-chain method was implemented, the previous month's work (in some sense) was replicated in about 4 days.

\section{Proof Methodology}

As was previously mentioned, the typical method of proof for theorems like Theorem \ref{premain} and \ref{main} is to accumulate enough primes to show $\sigma(N) > 2N$. In the context of \emph{optimal} solutions, this bound can be easier to surpass than the sum of the reciprocal of primes. Taking every allowable prime to be in \bb{T} without concern for special primes, and assuming each prime divides $N$ four times, we achieve $\sigma(N) > 2N$ after 47335 primes; 5 and every $1 \pmod{5}$ prime less than 2592521. A marginal improvement compared to our methodology. Though without any special knowledge regarding the number of times an arbitrary prime divides $N$, the sum of the reciprocal of primes may be an easier bound to surpass. Regardless, when nearly a million primes are needed for either method of proof, the choice largely becomes a matter of taste.

With regards to section \ref{ewellForm}. If one compares the improvement to the upper bound from increasing the number of times a prime divides $N$ from 4 to 14 (or 24), the improvements are minor. This begs the question, ``Why use this in our proof?'' The answer is simple. Since we are showcasing a different methodology for these types of proofs, we hope something here may spur innovation in someone else by broadening the view of the problem.

\section{Data Summary}

Occasionally, Mathematica would hang trying to show a number to be provably prime. This is why Module 17 has about half as much data as the other modules. The first time this happened, we started the next module, but after this instance, we skipped the seed prime, and restarted the module with the next number in the seed list.

Despite setting the recursion level to infinity, Mathematica did not seem to like using our algorithm after about 24,000 new seeds cycled through. To remedy this, we stopped after 20,000 iterations and called it a \emph{data module}. We have 30 data modules for the non-special primes using this algorithm. A proper remedy, namely taking the local variables within the program and making them global, solved this problem and allowed us to use as many as 50,000 seeds without issue in Modules 32-34.

Due to a lapse in programming, the primes that were $1 \pmod{12}$ and also $6 \pmod{7}$ were originally collected with the special primes. After the oversight was noticed, we separated them out and called them Module 31. The new primes were used as seeds for chains in Modules 32-34.

The filters resulting from Modules 1-18 were used to filter Module 20. As a result, 6305 primes were duplicated between Modules 19 and 20. The two data sets were combined for simplicity of filtering out duplicates.
\vspace{1 pc}

\begin{tabular}{|c|c|c|c|c|}
\multicolumn{5}{c}{Data Summary} \\\hline
Module & \# Primes & Total Primes & Module Sum & Total Sum \\\hline
1 & 29253   &   29253   & 0.5430147 &   0.5430147 \\\hline
2 & 27967   &   57220   & 0.0105602 &   0.5535749 \\\hline
3 & 27451   &   84671   & 0.0060902 &   0.5596651 \\\hline
4 & 27040   &  111711   & 0.0040347 &   0.5636998 \\\hline
5 & 27071   &  138782   & 0.0032882 &   0.5669880 \\\hline
6 & 26680   &   165462  & 0.0025726 &   0.5695606 \\\hline
7 & 26386   &   191848  & 0.0020041 &   0.5715647 \\\hline
8 & 26154   &   218002  & 0.0017367 &   0.5733014 \\\hline
9 & 26386   &   244388  & 0.0016088 &   0.5749102 \\\hline
10 & 26008  &   270396  & 0.0014174 &   0.5763276 \\\hline
11 & 25915  &   296311  & 0.0012916 &   0.5776192 \\\hline
12 & 25933  &   322244  & 0.0012246 &   0.5788438 \\\hline
13 & 25837  &   348081  & 0.0010565 &   0.5799003 \\\hline
14 & 25562  &   373643  & 0.0009254 &   0.5808257 \\\hline
15 & 25825  &   399468  & 0.0008905 &   0.5817163 \\\hline
16 & 25543  &   425011  & 0.0008013 &   0.5825176 \\\hline
17 & 12262  &   437273  & 0.0003992 &   0.5829168 \\\hline
18 & 25497  &   462770  & 0.0007536 &   0.5836704 \\\hline
19-20 & 44626 & 507396  & 0.0011789 &   0.5848493 \\\hline
21 & 25447  &   532843  & 0.0006398 &   0.5854891 \\\hline
22 & 25055  &   557898  & 0.0005856 &   0.5860747 \\\hline
23 & 25086  &   582984  & 0.0005580 &   0.5866328 \\\hline
24 & 25278  &   608262  & 0.0005657 &   0.5871985 \\\hline
25 & 25292  &   633554  & 0.0005443 &   0.5877428 \\\hline
26 & 25100  &   658654  & 0.0004953 &   0.5882381 \\\hline
27 & 25094  &   683748  & 0.0004879 &   0.5887260 \\\hline
28 & 25150  &   708898  & 0.0004763 &   0.5892022 \\\hline
29 & 25034  &   733932  & 0.0004517 &   0.5896539 \\\hline
30 & 25069  &   759001  & 0.0004470 &   0.5901010 \\\hline
31 & 41963  &   800964  & 0.0160788 &   0.6061797 \\\hline
32 & 36937  &   837901  & 0.0005961 &   0.6067758 \\\hline
33 & 54933  &   892834  & 0.0008298 &   0.6076056 \\\hline
34 & 66071  &   958905  & 0.0009650 &   0.6085706 \\\hline
Special & 2511 & 961416 & 0.0567245 &   0.6652951 \\
Primes  &      &        &           &             \\\hline
\end{tabular}

\bibliography{references}{}
\bibliographystyle{plain}

\end{document}